\documentclass[12pt,a4paper,english]{amsart}
\pdfoutput=1
\usepackage[T1]{fontenc}
\usepackage[latin9]{inputenc}
\usepackage{color}
\usepackage{babel}
\usepackage{amsthm}
\usepackage{amssymb}
\usepackage{enumitem}
\makeatletter

\pdfpageheight\paperheight
\pdfpagewidth\paperwidth

\numberwithin{equation}{section}
\numberwithin{figure}{section}
\usepackage{enumitem}		
\theoremstyle{plain}
\newtheorem{thm}{\protect\theoremname}
  \theoremstyle{definition}
  \newtheorem{example}[thm]{\protect\examplename}
  \theoremstyle{plain}
  \newtheorem{lem}[thm]{\protect\lemmaname}
 \theoremstyle{definition}
  \newtheorem{claim}[thm]{\protect\claimname}
\usepackage{fullpage}
\title{On the choosability of complete multipartite graphs with part size four}
\DeclareMathOperator*{\ch}{ch}
\usepackage{cite}
\DeclareMathOperator*{\ra}{ran} 

\makeatother

  \providecommand{\examplename}{Example}
  \providecommand{\lemmaname}{Lemma}
\providecommand{\theoremname}{Theorem}
 \providecommand{\claimname}{Claim}
 
 \usepackage{tikz}
 
\usepackage{fullpage}
\usetikzlibrary{shapes, fit, arrows, decorations.markings}
\usetikzlibrary{backgrounds}

\pgfdeclarelayer{bg}    
\pgfsetlayers{bg,main}  
\usepackage{pgf}

\begin{document}

\title{On the choice number of complete multipartite graphs with part size four}
\thanks{The first
 author thanks Institut Mittag-Leffler (Djursholm, Sweden) for the hospitality and creative environment.}

\author{H. A. Kierstead}
\address{School of Mathematical and Statistical Sciences\\ Arizona State University\\
Tempe, AZ 85287\\ USA} 
\email{kierstead@asu.edu}
\thanks{Research of the first 
 author is supported in part by NSA grant H98230-12-1-0212.}

 \author{ Andrew Salmon} 
 \address{6328 E. Cactus Wren Rd. \\Paradise Valley, AZ 85253 \\USA}
 \email{andrew.salmon@yale.edu}

\author{Ran Wang}
\address{School of Mathematical and Statistical Sciences\\ Arizona State University\\
Tempe, AZ 85287\\ USA}
\email{rwang31@asu.edu}



\begin{abstract}Let $\mathrm{ch}(G)$ denote the choice number of a graph $G$, and let $K_{s*k}$ be the complete $k$-partite graph with $s$ vertices in each
part. Erd\H{o}s, Rubin, and Taylor showed that $\mathrm{ch}( K_{2*k})=k$, 
and suggested the problem of determining the choice number of $K_{s*k}.$ The first author established
 $\mathrm{ch}( K_{3*k})=\left\lceil \frac{4k-1}{3}\right\rceil$. Here we prove $\mathrm{ch} (K_{4*k})=\left\lceil \frac{3k-1}{2}\right\rceil$.
\end{abstract}

\maketitle

\section{Introduction}

Let $G=(V,E)$ be a graph. A \textit{list assignment} $L$ for $G$ is a function
$L:V\rightarrow2^{\mathbb{N}}$, where $\mathbb{N}$ is the set of natural numbers
and $2^{\mathbb{N}}$ is the power set of $\mathbb{N}$. If $|L(v)|=k$ for all vertices
$v\in V$, then $L$ is a \textit{$k$-list assignment} for $G$. A\textit{n $L$-coloring}
$f$ from a list assignment $L$ is a function $f:V\rightarrow\mathbb{N}$ such that
$f(v)\in L(v)$ for all vertices $v\in V$ and $f(x)\ne f(y)$ whenever $xy\in E$.
$G$ is \textit{$L$-colorable} if there exists an $L$-coloring of $G$; it is \textit{$k$-choosable}
if it is $L$-choosable for all $k$-list assignments $L$. The \textit{list chromatic
number} or \textit{choi}ce\emph{ number} of $G$, denoted $\ch(G)$, is the smallest
integer $k$ such that $G$ is $k$-choosable. The general list coloring problem
may consider list assignments with uneven list sizes.

The study of list coloring was initiated by Vizing~\cite{Viz} and by Erd\H{o}s,
Rubin and Taylor \cite{Erdos79}. It is a generalization of two well studied areas
of combinatorics---graph coloring and transversal theory. Restricting the list assignment
to a constant function, yields ordinary graph coloring; restricting the graph to
a clique yields the problem of finding a system of distinct representatives (SDR)
for the family of lists. Both restrictions play a role in this paper.  Given the
general nature of this parameter, it is hardly surprising that there are not many 
graphs whose exact choice number is known. However, there are some amazingly elegant 
results that add to the subject's charm. For example, Thomassen \cite{Thomassen94} 
proved that planar graphs have choice number at most $5$, Voight \cite{Voigt93}
proved that this is tight, and Galvin \cite{Galvin95} proved that line graphs of
bipartite graphs have choice number equal to their clique number.

Erd\H{o}s et al.\ \cite{Erdos79} suggested determining the choice number of uniform
complete multipartite graphs. More generally, let $K_{1*k_{1},2*k_{2}\dots}$ denote
the complete multipartite graph with $k_{i}$ parts of size $i$, where zero terms
in the subscript are deleted. Since $K_{1*k}$ is a clique and $K_{s*1}$ is an independent
set, these cases are trivial. Alon \cite{Alon92} proved the general bounds $c_{1}k\log s\leq\ch(K_{s*k})\leq c_{2}k\log s$ for some constants $c_1,c_2>0$. This was tightened by Gazit and Krivelevich \cite{GK}.
\begin{thm}[Gazit and Krivelevich \cite{GK}]
$\ch(K_{s*k})=(1+o(1))\frac{\log s}{\log(1+1/k)}$. 
\end{thm}
The next well-known example provides the best lower bounds for small values of $s$. 
\begin{example}
$\ch(K_{s*k})\ge \lceil\frac{2(s-1)k-s+2}{s}\rceil$: Let $G=K_{s*k}$ have parts
$\{X_{1},\dots,X_{k}\}$ with $X_{i}=\{v_{i,1},\dots,v_{i,s}\}$. We will construct
an $(l-1)$-list assignment $L$ from which $G$ cannot be colored. Equitably partition
$C:=[2k-1]$ into $s$ parts $C_{1},\dots,C_{s}$. Define a list assignment $L$
for $G$ by $L(v_{i,j})=C\smallsetminus C_{j}$. Then each list has size at least
\[
2k-1-\left\lceil\frac{2k-1}{s}\right\rceil=\left\lfloor\frac{2ks-s-2k+1}{s}\right\rfloor=\left\lceil\frac{2(s-1)k-2s+2}{s}\right\rceil=l-1. 
\]
Consider any color $\alpha\in C$. Then $\alpha\in C_{i}$ for some $i\in[s]$. So
$\alpha\notin L(x_{i,j})$ for every $j\in[k]$. Thus any $L$-coloring of $G$ uses
at least two colors for every part $X_{j}$. Since vertices in distinct parts are
adjacent, they require distinct colors. As there are $k$ parts this would require
$2k>|C|$ colors, which is impossible.$ $
\end{example}
Restricting the question of Erd\"os et al., we ask for those integers $s$ such
that:
\begin{equation}
(\forall k\in\mathbb{Z}^{+})\left[\ch(K_{s*k})=l(s,k):=\left\lceil\frac{2(s-1)k-s+2}{s}\right\rceil\right].\label{qm}
\end{equation} 

The first two cases $s=2$ and $s=3$ have been solved:
\begin{thm}
[Erd\H{o}s, Rubin and Taylor \cite{Erdos79}]\label{thm:ERT}All positive integers
$k$ satisfy $ $$\ch(K_{2*k})=k$. 
\end{thm}

\begin{thm}[Kierstead \cite{Kie2000}]
\label{K-th}All positive integers $k$ satisfy $\ch(K_{3*k})=\lceil\frac{4k-1}{3}\rceil$.
\end{thm}
Recently, Kozik, Micek, and Zhu \cite{Kozik2014} gave a very different proof of
Theorem~\ref{K-th}. The following more general result appears in \cite{Ohba04}.
\begin{thm}[Ohba \cite{Ohba04}]
$\ch(K_{1*k_{1},3*k_{3}})=\max\{k,\lceil\frac{n+k-1}{3}\rceil\}$, where $k=k_{1}+k_{3}$
and $n=k_{1}+3k_{3}$. 
\end{thm}
The next example shows that the largest $s$ satisfying \eqref{qm} is at most
$14$.
\begin{example}
If $k$ is even then $\ch(K_{15*k})\geq l:=2k$: Let $G=K_{s*k}$ have parts $\{X_{1},\dots,X_{k}\}$
with $X_{i}=\{v_{i,1},\dots,v_{i,s}\}$. We will construct an $(l-1)$-list assignment
$L$ from which $G$ cannot be colored. Equitably partition $C:=[3k-1]$ into $6$
parts $C_{1},\dots,C_{6}$, and fix a bijection $f:[15]\rightarrow\binom{[6]}{2}$.
Define a list assignment $L$ for $G$ by 
\[
L(v_{i,j})=C\smallsetminus\bigcup\{C_{h}:h\in f(i)\}. 
\]
Then each list has size at least 
\[
3k-1-2\left\lceil\frac{3k-1}{6}\right\rceil=2k-1=l-1. 
\]
Consider any two colors $\alpha,\beta\in C$. Then $\alpha,\beta\in\bigcup\{C_{h}:h\in f(i)\}$
for some $i\in[15]$. So $\alpha,\beta\notin L(x_{i,j})$ for every $j\in[k]$. Thus
any $L$-coloring of $G$ uses at least three colors for every part $X_{j}$. Since
$3k>|C|$, this is impossible.
\end{example}

Yang \cite{Yang2003} proved $\lceil\frac{3k}{2}\rceil\le\ch(K_{4*k})\le\lceil\frac{7k}{4}\rceil$,
and Noel et al. \cite{Noel2013} improved the upper bound to $\lceil\frac{5k-1}{3}\rceil$.
  The main result of this paper is that \eqref{qm} holds for $s=4$. To prove this theorem 
 we first extract a simple proof of Theorem~\ref{K-th} from \cite{Noel2013}, and then
elaborate on it.

\begin{thm}
\label{M-th}$\ch(K_{4*k})=l(4,k):=\lceil\frac{3k-1}{2}\rceil$. 
\end{thm}

Some of the recent development of list coloring of complete multipartite graphs
has been motivated by paintability, or on-line choosability. Introduced by Schauz \cite{Schauz09},
\emph{paintability} is a coloring game played between two players Alice and Bob on a graph $G=(V,E)$ and a function $f:V\rightarrow \mathbb N$. Let $V_{i}$ denote the vertex set at the start of
round $i$; so $V_{1}=V$. At  round $i$, Alice selects a
nonempty set of vertices $A_{i}\subseteq V_{i}$, and Bob selects an independent
set $B_{i}\subseteq A_{i}$. Then $B_{i}$ is  deleted from the graph so that $V_{i+1}=V_{i}\smallsetminus B_{i}$,
and the rounds are continued until $V_{n}=\emptyset$. Alice's goal is to present some vertex $v$ more than $f(v)$ times, while Bob's goal is to choose every vertex before it has been presented $f(v)+1$ times. 
We say that $G$ is \emph{on-line $f$-choosable} if player $B$ has a strategy
such that any vertex $v\in V$ is in at most $f(v)$ sets $A_{i}$, and \emph{on-line
$k$ choosable} if $G$ is on-line $f$-choosable when $f(v)=k$ for all $v\in V$.
The \emph{on-line choice number}, denoted $\ch^{OL}(G)$, is the least $k$ such
that $G$ is on-line $k$-choosable.

This game formulation hides the on-line nature of the problem.  Another way of thinking about it is that Alice has secretly assigned lists of colors to all the vertices.  At round $i$ she reveals all vertices whose list contains color $i$, and Bob colors an independent set of them with color $i$. In this formulation it is clear that $\ch(G)\le\ch^{OL}(G)$.

Surprisingly, Schauz \cite{Schauz09} proved that many  results on choice number, including Brooks' theorem, Thomassen's theorem,
and the Bondy-Boppana kernel lemma carry over to on-line choice number.
It is unknown whether $\ch^{OL}(G)-\ch(G)$ is bounded by a constant. Indeed, no graphs
are known for which $\ch^{OL}(G)-\ch(G)\ge2$. It is known that 
\[\ch(K_{2,2,3})=3<4={\ch}^{OL}(K_{2,2,3}).\]
The explicit value of $\ch(K_{4*k})$
provided by Theorem~\ref{M-th} may be useful for establishing larger gaps. In Section~4 we show that $\ch(K_{4*3})<{\ch}^{OL}(K_{4*3})$.

\section{Set-up}

Fix $s,k\in\mathbb{Z}^{+}$. Let $G=(V,E)=K_{s*k}$, and $\mathcal{P}$ be the partition
of $V$ into $k$ independent $s$-sets. Let $l=l(k,s)=\lceil\frac{(s-1)2k-s+2}{s}\rceil$,
and consider any $l$-list assignment $L$ for $G$.
 Put $C^*=\bigcup_{x\in V}L(x)$. 
Let $L\neg\alpha$ be the result of deleting $\alpha$ from every list of $L$. 

We may write $x_{1}\dots x_{t}$ for the subpart $S=\{x_{1},\dots,x_{t}\}\subseteq X\in\mathcal{P}$;
when we use this notation we implicitly assume the $x_{i}$ are distinct. Also set
$\overline{S}=X\smallsetminus S$. For a set of verties $S\subseteq V$ let $\mathcal{L}(S)=\{L(x):x\in S\}$,
$L(S)=\bigcap\mathcal{L}(S)$, $W(S)=\bigcup\mathcal{L}(S)$, and $l(S)=|L(S)|$.
The operation of replacing the vertices in $S$ by a new vertex $v_{S}$ with the same neighborhood as $S$ is called
\emph{merging. }The new vertex $v_{S}$ is said to be \emph{merged}; vertices that
are not merged are called \emph{original. }When merging a set $S$ we also
create a list $L(v_{S})=L(S)$.

For a color $\alpha\in C^*$, let $|X,\alpha|=|\{x\in X:\alpha\in L(x)\}|$ be the
number of times $\alpha$ appears in the lists of vertices of $X$, $N_{i}(X)=\{\alpha\in C^*:|X,\alpha|=i\}$
be the set of colors that appear exactly $i$ times in the lists of vertices in $X$,
$n_{i}(X)=|N_{i}(X)|$, and $N(X)=N_{2}(X)\cup N_{3}(X)$. Let $\sigma_{i}(X)=\sum\{l(I):I\subseteq X\wedge|I|=i\}$
and $\mu_{i}(X)=\max\{l(I):I\subseteq X\wedge|I|=i\}$.

For a set $S$ and element $x$ we use the notation $S+x=S\cup \{x\}$ and $S-x=S\smallsetminus \{x\}$.

The following lemma was proved independently by Kierstead \cite{Kie2000}, and by
Reed and Sudakov \cite{Reed02}, \cite{ReedS04}, and named by Rabern.
\begin{lem}
[Small Pot Lemma]\label{spl}If $\ch(G)>r$ then there exists a list assignment
$L$ such that $G$ has no $L$-coloring, all lists have size  $r$, and their union
has size less than $|V(G)|$. 
\end{lem}
If $s$ does not satisfy \eqref{qm} then there is a minimal counterexample $k$
with $\ch(K_{s,k})>l(s,k)$. By the Small Pot Lemma, this is witnessed by a list
assignment $L$ with $|\bigcup\{L(x):x\in V(G)\}|<|V|$. We always assume $L$ has
this property.
\begin{lem}
\label{ns=00003D0}Every part $X$ of $G$ satisfies $L(X)=\emptyset$. \end{lem}
\begin{proof}
Otherwise there exists a list assignment $L$, a color $\alpha$, and a part $X$
such that $\alpha\in L(X)$. Color each vertex in $X$ with $\alpha$, set $G'=G-X$,
and put $L'=L\neg\alpha$. Then $L'$ witnesses that $k-1$ is a smaller counterexample, 
a contradiction.
\end{proof}
By Lemma~\ref{ns=00003D0}, $n_{s}(X)=0$ for each part $X\in\mathcal{P}$. So by
the Small Pot Lemma, $|W(X)|=\sum_{i=1}^{s-1}n_{i}(X)<sk$. Also $\sum_{i=1}^{s-1}i\, n_{i}(X)=sl$ 
is the number of occurrences of colors in the lists of vertices of $X.$ Thus 
\begin{alignat}{1}
\sum_{i=2}^{s-1}(i-1)n_{i}(X)\geq sl-|W(X)|\geq s(l-k)+1.\label{spb} 
\end{alignat}
Now we warm-up by giving a short proof extracted from \cite{Noel2013} of $ $Theorem~\ref{K-th}.
\begin{proof}
[Proof of Theorem \ref{K-th}]Let $s=3$, $l=l(3,k)$, and assume $G$ is a 
counterexample with $k$ minimal. Then $k>1$. By Lemma~\ref{ns=00003D0}, $n_{3}(X)=0$ for all $X\in\mathcal{P}$.
We obtain a contradiction by $L$-coloring $G$. First we use the following steps
to partition $V$ into sets of vertices that will receive the same color. Then we
\emph{merge} each set $I$ into a single vertex $v_{I}$, and assign $v{}_{I}$ the
set of colors in $L(I)$. Finally we apply Hall's Theorem to chose a system of distinct
representatives (SDR) for these new lists; this induces an $L$-coloring of $G$. 
\begin{enumerate}[label=\textbf{Step \arabic*.}, wide=\parindent,labelindent=0pt]
\item Partition $\mathcal{P}$ into a set  $\mathcal{R}$ of  $l-k$ \emph{reserved} parts together with a set  $\mathcal{U}=\mathcal{P}\smallsetminus\mathcal{R}$ of $2k-l$
\emph{unreserved} parts.
\item Choose $\mathcal{U}_{1}\subseteq\mathcal{U}$ maximum subject to $|\mathcal{U}_{1}|\leq\mu_{2}(X)$
for all $X\in\mathcal{U}_{1}$, and subject to this, $\nu=\sum_{X\in\mathcal{U}_{1}}\mu_{2}(X)$
is maximum. Set $u_{1}=|\mathcal{U}_{1}|$. For each $X\in\mathcal{U}_{1}$ choose
$ $a pair $I_{X}\subseteq X$ with $l(I_{X})\geq u_{1}$ maximum. Put $\mathcal{U}_{2}=\mathcal{U}\smallsetminus\mathcal{U}_{1}$
and $u_{2}=|\mathcal{U}_{2}|$. So
\begin{equation}
\mbox{if }u_{1}<2k-l\mbox{ then }\mu_{2}(X)\leq u_{1}\mbox{ for all }X\in\mathcal{U}_{2},\label{U1}
\end{equation}
since otherwise we could increase $\nu$ by adding $X$ to $\mathcal{U}_{1}$, and
deleting one part $Y\in\mathcal{U}_{1}$ with $\mu_{2}(Y)=u_{1}$, if such a part
$Y$ exists.
\item Using \eqref{spb}, each part $X\in\mathcal{P}$ satisfies 
\[
n_{2}(X)\geq3(l-k)+1\geq3\left\lceil\frac{k-1}{3}\right\rceil+1\geq k-1+1=k. 
\]
Form an SDR $f$ for $\{L(v_{I_{X}}):X\in\mathcal{U}_{1}\}\cup\{N(X):X\in\mathcal{R}\}$
by greedily choosing representatives for the first family and then for the second
family. For each $X\in\mathcal{R}$ choose a pair $I_{X}\subseteq X$ so that $f(x)\in L(I_{X})$. 
\item For each $X\in\mathcal{U}_{1}\cup\mathcal{R},$ merge $I_{X}$ to a new vertex $v_{I_X}$,
let $z_{X}\in X\smallsetminus I_{X}$, and set $X'=\{v_{I_X},z_{X}\}$. If $X\in\mathcal{U}_{2}$,
set $X'=X$. This yields a graph $G'$ with parts $\mathcal{P}'=\{X':X\in\mathcal{P}\}$,
and list assignment $L$. 
\end{enumerate}
Next we use Hall's Theorem to prove that $\{L(x):x\in V(G')\}$ has an SDR. For this it suffices to prove:
\begin{equation}
|S|\leq\left|\bigcup\{L(x):x\in S\}\right|\mbox{ for every }S\subseteq V(G').\label{HC}
\end{equation}
To prove \eqref{HC}, let $S\subseteq V(G')$ be arbitrary, and set $W=W(S):=\bigcup\{L(x):x\in S\}$.
We consider several cases in order, always assuming all previous cases fail.
\begin{enumerate}[label=\textbf{Case \arabic*:}, wide=\parindent,labelindent=0pt ]
\item There exists $X\in\mathcal{P}$ with $|S\cap X'|=3$. Then $|S|\leq2k+u_{2}$, $X'=X\in\mathcal{U}_{2}$
and $u_{2}\geq1$. Thus $u_{1}\leq2k-l-u_{2}<2k-l$, and so by \eqref{U1}, $u_{1}\geq\mu_{2}(X)\ge\sigma_{2}(X)/3$.
Using inclusion-exclusion, and Lemma~\ref{ns=00003D0}, 
\begin{eqnarray*}
|W| & \geq & |W(X)|\geq\sigma_{1}(X)-\sigma_{2}(X)+\sigma_{3}(X)\geq3l-3u_{1}=3l-3(2k-l-u_{2})\\
 & \geq & 6(l-k)+3u_{2}\geq(2k-2)+(2+u_{2})\geq2k+u_{2}\geq|S|.
\end{eqnarray*}

\item There is $X\in\mathcal{U}_{2}$ with $|S\cap X'|=2$. Then $X=X'$ and $|S|\leq2k.$
Since $u_{1}=2k-l-u_{2}<2k-l$, \eqref{U1} yields
\[
|W|\geq|W(S\cap X)|\geq2l-l(S\cap X)\geq2l-u_{1}\geq2l-(2k-l-u_{2})\geq3l+1-2k=2k\geq|S|.
\]

\item There is $X\in\mathcal{U}_{1}$ with $|S\cap X'|=2$. As $|S|\leq2k-u_{2}=l+u_{1}$
and $L(v_{I_X}z_{X})=L(X)=\emptyset$, 
\[
|W|\geq|W(S\cap X')|\geq l(v_{I_X})+l(z_{X})-l(v_{I_X}z_{X})\geq u_{1}+l\geq|S|. 
\]

\item $S$ has an original vertex. Then $|S|\leq l\leq|W|$.
\item All vertices of $S$ have been merged. Then $|S|\le|f(S)|\leq |W|$. 
\end{enumerate}
\end{proof}

\section{The main theorem}

In this section we prove our main result, Theorem~\ref{M-th}. The case when $k$
is odd is considerably more technical. Casual or first time readers may wish to avoid
these additional details; the proof is organized so that this is possible. In particular,
in the even case Step 11 and Lemmas~\ref{HC-lem} and \ref{HC'-lem} are not involved.
We often use the partition $k=(2k-l)+(l-k)$ of the integer $k$, and note that $2k-l=l-k+b$,
where $b=k\bmod2$.
\begin{proof}
[Proof of Theorem \ref{M-th}]Our set-up is the same as in the proof of Theorem~\ref{K-th}.
Let $s=4$, $l=l(4,k)$, and $G=K_{4*k}$. The theorem is trivial if $k=1$. Let
$k>1$ be a minimal counterexample, and let $L$ be an $l$-list assignment for $G$
with $|W(V)|\leq4k-1$ and $L(X)=\emptyset$ for all parts $X\in\mathcal{P}$. Again
we partition $V$ into sets of vertices that will receive the same color, and then
find an SDR for the induced list assignment that in turn induces an $L$-coloring
of $G$. See Figure \ref{figure}.

\begin{figure}
\begin{tikzpicture}[scale=.90, 
	dot/.style={fill,circle,inner sep=0pt,minimum size=7pt}]
	

\node[dot] (a) at (0.5,0.75) {};
\node[dot] (b) at (0.5,1.75) {};
\node[dot] (c) at (0.5,3.25) {};
\node[dot] (d) at (0.5,4.25) {};
\node[draw, fit=(a) (b), inner sep=1pt, ellipse] {};
\node[draw, fit=(c) (d), inner sep=1pt, ellipse] {};
\draw [dashed] (0,2.5) -- (2.5,2.5);
\draw (2.5, 0) -- (2.5, 5);


\node[dot] (e) at (3,2.25) {};
\node[dot] (f) at (3,3.25) {};
\node[dot] (g) at (3,4.25) {};
\node[dot] at (3,0.75) {};
\node[draw, fit=(e) (f) (g), inner sep=1pt, ellipse] {};
\draw [dashed] (2.5,1.5) -- (5,1.5);
\draw (5, 0) -- (5, 5);


\node[dot]  at (5.5,0.75) {};
\node[dot]  at (5.5,1.75) {};
\node[dot] (h) at (5.5,3.25) {};
\node[dot] (i) at (5.5,4.25) {};
\node[draw, fit=(h) (i), inner sep=1pt, ellipse] {};
\draw [dashed] (5,2.5) -- (7.5,2.5);
\draw (7.5, 0) -- (7.5, 5);

\draw [ultra thick] (9.5, -1) -- (9.5, 6);
\node[dot]  at (8,0.75) {};
\node[dot]  at (8,1.75) {};
\node[dot]  at (8,3.25) {};
\node[dot]  at (8,4.25) {};


\node[dot] (j) at (10,2.25) {};
\node[dot] (k) at (10,3.25) {};
\node[dot] (l) at (10,4.25) {};
\node[dot] at (10,0.75) {};
\node[draw, fit=(j) (k) (l), inner sep=1pt, ellipse] {};
\draw [dashed] (9.5,1.5) -- (12,1.5);
\draw (12, 0) -- (12, 5);


\node[dot] (m) at (12.5,0.75) {};
\node[dot] (n) at (12.5,1.75) {};
\node[dot] (o) at (12.5,3.25) {};
\node[dot] (p) at (12.5,4.25) {};
\node[draw, fit=(m) (n), inner sep=1pt, ellipse] {};
\node[draw, fit=(o) (p), inner sep=1pt, ellipse] {};
\draw [dashed] (12,2.5) -- (14.5,2.5);
\draw (14.5, 0) -- (14.5, 5);


\node[dot]  at (15,0.75) {};
\node[dot]  at (15,1.75) {};
\node[dot] (q) at (15,3.25) {};
\node[dot] (r) at (15,4.25) {};
\node[draw, fit=(q) (r), inner sep=1pt, ellipse] {};
\draw [dashed] (14.5,2.5) -- (17,2.5);

\node[draw=none,fill=white] at (1.25,-.5) {$\mathcal{U}_1$};
\node[draw=none,fill=white] at (3.75,-.5) {$\mathcal{U}_2$};
\node[draw=none,fill=white] at (6.25,-.5) {$\mathcal{U}_3$};
\node[draw=none,fill=white] at (8.5,-.5) {$\mathcal{U}_4$};
\node[draw=none,fill=white] at (10.75,-.5) {$\mathcal{R}_1$};
\node[draw=none,fill=white] at (13.25,-.5) {$\mathcal{R}_2$};
\node[draw=none,fill=white] at (15.75,-.5) {$\mathcal{R}_3$};

\node[draw=none,fill=white] at (1.25,5.5) {$u_1$};
\node[draw=none,fill=white] at (3.75,5.5) {$u_2$};
\node[draw=none,fill=white] at (6.25,5.5) {$u_3$};
\node[draw=none,fill=white] at (8.5,5.5) {$u_4$};
\node[draw=none,fill=white] at (10.75,5.5) {$r_1$};
\node[draw=none,fill=white] at (13.25,5.5) {$r_2$};
\node[draw=none,fill=white] at (15.75,5.5) {$r_3$};

\node[draw=none,fill=white] at (4.75,6) {$2k-l$};
\node[draw=none,fill=white] at (13.25,6) {$l-k$};

\begin{pgfonlayer}{bg}
	\path [fill=gray!30] (0,0) rectangle (2.5,5);
	\path [fill=gray!30] (2.5,1.5) rectangle (5,5);
	\path [fill=gray!30] (5,2.5) rectangle (7.5,5);
	\path [fill=gray!30] (9.5,1.5) rectangle (12,5);
	\path [fill=gray!30] (12,0) rectangle (14.5,5);
	\path [fill=gray!30] (14.5,2.5) rectangle (17,5);
	\draw (0,0) rectangle (17,5);

	
	\draw [arrows={|<->|}] (0,5.5) -- (2.5,5.5);
	\draw [arrows={|<->|}] (2.5,5.5) -- (5,5.5);
	\draw [arrows={|<->|}] (5,5.5) -- (7.5,5.5);
	\draw [arrows={|<->|}] (7.5,5.5) -- (9.5,5.5);
	\draw [arrows={|<->|}] (9.5,5.5) -- (12,5.5);
	\draw [arrows={|<->|}] (12,5.5) -- (14.5,5.5);
	\draw [arrows={|<->|}] (14.5,5.5) -- (17,5.5);
	
	\draw [arrows={|<->|}] (0,6) -- (9.5,6);
	\draw [arrows={|<->|}] (9.5,6) -- (17,6);
\end{pgfonlayer}
\end{tikzpicture}
\caption{The partition $\mathcal P'$ of $K_{4*k}$.\label{figure}}
\end{figure}
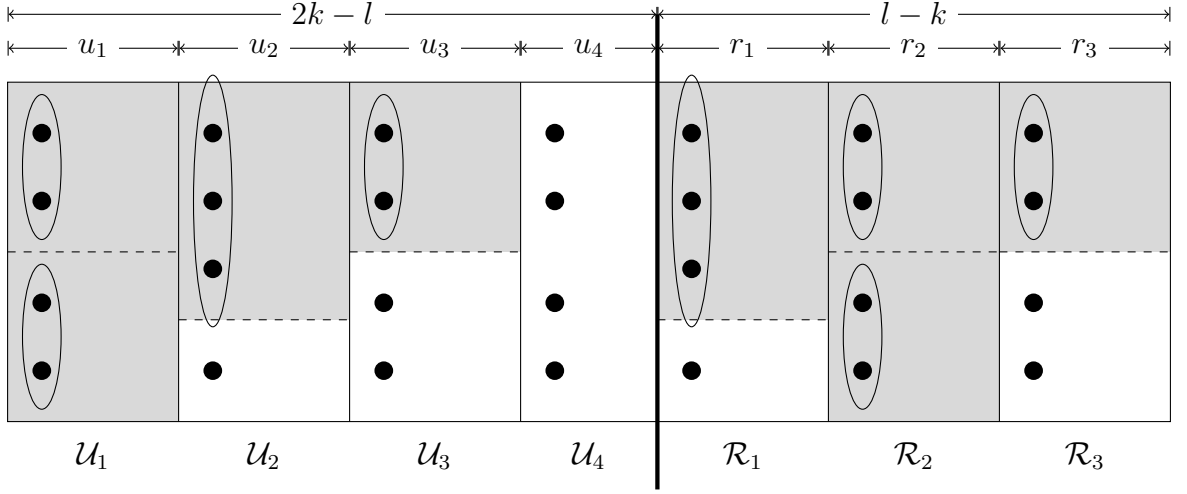

\begin{enumerate}[label=\textbf{Step \arabic*.},ref=Step \arabic*, wide=\parindent,labelindent=0pt]
\item Partition $V$ as $\mathcal{P}=\mathcal{U}\cup\mathcal{R}$, where $|\mathcal{R}|=l-k$,
$|\mathcal{U}|=2k-l$, $\mathcal{R}=\mathcal{R}_{1}\cup\mathcal{R}_{2}\cup\mathcal{R}_{3}$
and $\mathcal{U}=\mathcal{U}_{1}\cup\mathcal{U}_{2}\cup\mathcal{U}_{3}\cup\mathcal{U}_{4}$
as follows.
\item Choose $\mathcal{U}_{1}\subseteq\mathcal{P}$ maximum subject to $|\mathcal{U}_{1}|\leq2k-l$
and for every $X\in\mathcal{U}_{1}$ there is a pair $I_{X}\subseteq X$ with $l(I_{X}),l(\overline{I}_{X})\geq k$.
Put $\mathcal{U}_{1}\subseteq\mathcal{U}$, and let $u_{1}:=|\mathcal{U}_{1}|$.
Then: 
\begin{equation}
\mbox{If }u_{1}<2k-l\mbox{ then }(\forall X\in\mathcal{P}\smallsetminus\mathcal{U}_{1})(\forall I\subseteq X)[|I|=2\rightarrow \min\{l(I),l(\overline{I})\}\leq k-1].\label{U-1-eq} 
\end{equation}

\item Choose $\mathcal{U}_{2}\subseteq\mathcal{P}\smallsetminus\mathcal{U}_{1}$ maximum
subject to $|\mathcal{U}_{2}|\leq2k-l-u_{1}$ and $|\mathcal{U}_{2}|\leq\mu_{3}(X)$
for all $X\in\mathcal{U}_{2}$; subject to this let $\nu=\sum_{X\in\mathcal{U}_{2}}\mu_{3}(X)$
be maximum. Put $\mathcal{U}_{2}\subseteq\mathcal{U}$, and let $u_{2}=|\mathcal{U}_{2}|$.
If $\mathcal{U}_{2}\ne\emptyset$ then let $\dot{Z}\in\mathcal{U}_{2}$; else $\dot{Z}=\emptyset$.
For each $X\in\mathcal{U}_{2}$ choose a triple $I_{X}\subseteq X$ with $l(I_{X})\geq u_{2}$
maximum. Since $\nu$ cannot be increased: 
\begin{equation}
\mbox{If }u_{1}+u_{2}<2k-l\mbox{ then }(\forall X\in\mathcal{U}_{3}\cup\mathcal{U}_{4}\cup\mathcal{R})[\mu_{3}(X)\leq u_{2}].\label{U2}
\end{equation}

\item Choose $\mathcal{R}_{1}\subseteq\mathcal{P}\smallsetminus(\mathcal{U}_{1}\cup\mathcal{U}_{2})$
maximum subject to $|\mathcal{R}_{1}|\leq l-k$ and for all $X\in\mathcal{R}_{1}$
there exists $I_{X}\subseteq X$ with $|I_{X}|=3$ such that there is an SDR $f_{1}$
of $\mathcal{L}(M_{1})$, where $M_{1}:=\{v_{I_{X}}:X\in\mathcal{U}_{2}\cup\mathcal{R}_{1}\}$;
let $C_{1}=\ra(f_{1})$. Put $\mathcal{R}_{1}\subseteq\mathcal{R}$, and let $r_{1}:=|\mathcal{R}_{1}|$.
Then:
\begin{equation}
\mbox{If }r_{1}<l-k\mbox{ then }(\forall X\in\mathcal{U}_{3}\cup\mathcal{U}_{4}\cup\mathcal{R}_{2}\cup\mathcal{R}_{3})[N_{3}(X)\subseteq C_{1}].\label{R1}
\end{equation}

\item Choose $\mathcal{U}_{3}\subseteq\mathcal{P}\smallsetminus(\mathcal{U}_{1}\cup\mathcal{U}_{2}\cup\mathcal{R}_{1})$
maximum subject to $|\mathcal{U}_{3}|\leq2k-l-u_{1}-u_{2}$ and $l-k+u_{2}+|\mathcal{U}_{3}|\leq\mu_{2}(X)$
for all $X\in\mathcal{U}_{3}$; subject to this let $\nu=\sum_{X\in\mathcal{U}_{3}}\mu_{2}(X)$
be maximum. Put $\mathcal{U}_{3}\subseteq\mathcal{U}$, and $u_{3}=|\mathcal{U}_{3}|$.
Since $\nu$ cannot be increased: 
\begin{equation}
\mbox{If }u_{1}+u_{2}+u_{3}<2k-l\mbox{ then }(\forall X\in\mathcal{U}_{4}\cup\mathcal{R}_{2}\cup\mathcal{R}_{3})[\mu_{2}(X)\leq l-k+u_{2}+u_{3}].\label{U3}
\end{equation}

For all $X\in\mathcal{U}_{3}$ choose a pair $I_{X}=xy\subseteq X$ with $l(I_{X})\ge l-k+u_{2}+u_{3}$
maximum; subject to this choose $I_{X}$ so that $\Delta_{1}(I_{X}):=l(I_{X})-l(\overline{I}_{X})$
is maximum.
 Set  $\Delta_{2}(I_{X}):=2u_{2}-l(xyz)-l(xyw)$
, where $zw=\overline{I}_{X}$. $ $ Using $r_1\leq l-k$,
extend $f_{1}$ to an SDR $f_{2}$ of $\mathcal{L}(M_{2})$, where $M_{2}:=M_{1}\cup\{v_{I_{X}}:X\in\mathcal{U}_{3}\}$;
set $C_{2}=\ra(f_{2})$. 

\item \begin{flushleft}
Choose $\mathcal{R}_{2}\subseteq\mathcal{P}\smallsetminus(\mathcal{U}_{1}\cup\mathcal{U}_{2}\cup\mathcal{U}_{3}\cup\mathcal{R}_{1})$
maximum subject to $|\mathcal{R}_{2}|\leq l-k-r_{1}$ and $\sigma_{2}(X)-\sigma_{3}(X)\ge5(l-k)+2u_{1}+2u_{2}+u_{3}+r_{1}+|\mathcal{R}_{2}|$
for all $X\in\mathcal{R}_{2}$; subject to this let $\sum_{X\in\mathcal{R}_{2}}\sigma_{2}(X)-\sigma_{3}(X)$
be maximum. Put $\mathcal{R}_{2}\subseteq\mathcal{R}$, and set $r_{2}=|\mathcal{R}_{2}|$.
Then:
\begin{flalign}
\mbox{If }r_{1}+r_{2} & <l-k\mbox{ then }(\forall X\in\mathcal{U}_{4}\cup\mathcal{R}_{3})\nonumber \\
 & [\sigma_{2}(X)-\sigma_{3}(X)\leq5(l-k)+2u_{1}+2u_{2}+u_{3}+r_{1}+r_{2}].\label{R2}
\end{flalign}

\par\end{flushleft}
\item Choose $\mathcal{R}_{3}\subseteq\mathcal{P}\smallsetminus(\mathcal{U}_{1}\cup\mathcal{U}_{2}\cup\mathcal{U}_{3}\cup\mathcal{R}_{1}\cup\mathcal{R}_{2})$
with $|\mathcal{R}_{3}|=l-k-r_{1}-r_{2}$, and set $\mathcal{R}=\mathcal{R}_{1}\cup\mathcal{R}_{2}\cup\mathcal{R}_{3}$.
Let $r_{3}=|\mathcal{R}_{3}|$. For $I\subseteq X$, put $L'(I)=L(I)\smallsetminus C_{2}$
and $l'(I)=|L'(I)|$. Using Lemma~\ref{R3-L}, for all $X\in\mathcal{R}_{3}$ there
exists a pair $I_{X}\subseteq X$ with $l'(\overline{I}_{X})\leq l'(I_{X})$ such
that $f_{2}$ can be extended to an SDR $f_{3}$ of $\mathcal{L}(M_{3})$, where
$M_{3}:=M_{2}\cup\{v_{I_{X}}:X\in\mathcal{R}_{3}\}$. Let $C_{3}=\ra(f_{3})$. 
\item Put $\mathcal{U=}\mathcal{P}$$\smallsetminus\mathcal{R}$, $\mathcal{U}_{4}:=\mathcal{U}\smallsetminus(\mathcal{U}_{1}\cup\mathcal{U}_{2}\cup\mathcal{U}_{3})$,
and $u_{4}:=|\mathcal{U}_{4}|$.
\item Using Lemma~\ref{R2-L}, choose a pair $I_{X}\subseteq X$ for all $X\in\mathcal{R}_{2}$
so that $\mathcal{L}(M_{4})$ has an SDR $f_{4}$ extending $f_{3}$, where $M_{4}:=M_{3}\cup\{v_{I_{X}},v_{\overline{I}_{X}}:X\in\mathcal{U}_{1}\cup\mathcal{R}_{2}\}$.
\item Let $G':=(V',E')$ be the graph obtained from $G$ by merging each $I_{X}$ with
$X\in\mathcal{U}_{2}\cup\mathcal{U}_{3}\cup\mathcal{U}_{1}\cup\mathcal{R}_{1}\cup\mathcal{R}_{2}\cup\mathcal{R}_{3}$
 and each $\overline{I}_{X}$ with $X\in\mathcal{U}_{1}\cup\mathcal{R}_{2}$. For a part $X$, let $X'$ be the corresponding part in
$G'$, and set $\mathcal{P}'=\{X':X\in\mathcal{P}\}$. 
\item \label{extra-Step} Set $0=\dot u=\dot r=\ddot u$. If $k$ is odd ($b=1$) then
we merge one more pair of vertices under any of the following special circumstances: 

\begin{enumerate}
\item there exists $X\in\mathcal{U}_{4}$ with $|W(X)|<|G'|$. Fix
such an $X=\dot{X}$. By Lemma~\ref{HC-lem}, $r_3=0$ and there is a pair $\dot{I}\subseteq\dot{X}$
such that (i) $f_{4}$ can be extended to an $SDR$ $f$ of $\mathcal{L}(M)$, where
$M:=M_{4}+v_{\dot{I}}$; (ii) $|W(\{v_{\dot{I}},v\})|\geq2k-1$, and if equality
holds then $|W(\{v_{\dot{I}},v\}\cup \dot{Z}')\cup C_4|\geq2k$ for both $v\in\overline{\dot{I}}$;
and (iii) $W(\overline{\dot{I}}+v_{\dot{I}})\geq|G'|-1$. Merge $\dot{I}$ and set 
$\dot{u}=1$.
\item $u_{1}=r_{2}=0$ and there is $ $$Y\in\mathcal{R}_{3}$ with $|W(Y)|\leq3k-1-u_{2}-r_{1}$. Then (a) fails since $r_3\ge 1$. 
Fix such a $Y=\dot{Y}$. As $u_{1}=0=r_{2}$, $M_{4}=M_{3}$. Since $r_3\ne0$, (a) is not executed. By Lemma~\ref{R3-L},
$f_{3}$ can be chosen so that it is an SDR $ $ of $\mathcal{L}(M)$, where $M:=M_{4}+v_{\overline{I}_{\dot{Y}}}$.
Merge $\overline{I}_{\dot{Y}}$ and set
$\dot{r}=1$.
\item condition (a) fails and there exist $X\in\mathcal{U}_{4}$ and $xyz\subseteq X$
with 
\[
|W(xyz\cup\dot Z')|\leq2k+u_{4}-1<|W(X)|.
\]
Fix such an $X=xyzw=\ddot{X}$. By Lemma~\ref{HC'-lem} there is a pair $\ddot{I}\subseteq xyz$
such that (i) $f_{4}$ can be extended to an SDR $f$ of $\mathcal{L}(M)$, where
$M:=M_{4}+v_{\ddot{I}}$; (ii) $|W(\{v_{\ddot{I}},v\})|\geq2k$ for $v\in xyz\smallsetminus\ddot{I}$
and $|W(\{v_{\ddot{I}},w\})|\geq2k-1$; and (iii) $|W(\overline{\ddot{I}}+v_{\ddot{I}})|\geq2k+u_{4}$.
Merge $I_{\ddot{X}}:=\ddot{I}$ and  set 
$\ddot{u}=1$.
\end{enumerate}

\item Recall that $G'$ is the graph obtained after the first ten steps. Let $H$ be the
final graph obtained by this merging procedure. (If $b=0$, and possibly otherwise,
$H=G'$). Also let $M$ be the final set of merged vertices, $f$ be the final SDR
of $\mathcal{L}(M)$, and $C=\ra(f)$. 
\end{enumerate}

Our next task is to state and prove the four lemmas on which the algorithm is based. We will need the following easy claim.

\begin{claim}\label{cl}
Let  $\mathcal P_1,\mathcal P_2,\mathcal P_3$ be the three partitions of a $4$-set   $X$ into pairs. For all  $I_1\in \mathcal P_1,I_2\in \mathcal P_2,I_3\in \mathcal P_3$ there exists  $v\in X$ such that either (i)  $v\in I_1\cap I_2\cap I_3$ or (ii)  $v\notin I_1 \cup I_2\cup I_3$.
\end{claim}

\begin{lem}
\label{R3-L}There is a family $\mathcal{I=}\{I_{X}:X\in\mathcal{R}_{3}\}$ such
that $I_{X}\subseteq X$, $|I_{X}|=2$, $l'(I_{X})\geq l'(\overline{I}_{X})$, and
$\mathcal{L}(M_{2}\cup\{v_{I_{X}}:X\in\mathcal{R}_{3}\})$ has an SDR $f_{3}$ extending
$f_{2}$. 

Furthermore, if  $u_{1}=0=r_{2}$ and there is $\dot Y\in\mathcal{R}_{3}$ with $|W(\dot Y)|\leq3k-1-u_{2}-r_{1}$,
 then $I_{\dot{Y}}$
can be chosen so that there is an SDR $f$ of $\mathcal{L}(M)$ extending $f_{2}$,
where $M=M_{3}+v_{\overline{I}_{\dot{Y}}}$. \end{lem}
\begin{proof}
Consider any $X\in\mathcal{R}_{3}$, and let $A(X)=N_{2}(X)\smallsetminus C_{2}$
be the set of colors available for coloring a pair of vertices from $X$. Then $L'(I)=L(I)\cap A(X)$
for all pairs $I\subseteq X$. For each color $\alpha\in A$, set $I(\alpha)=\{x\in X:\alpha\in L(x)\}$.
As $A(X)\subseteq N_{2}(X)$, $|I(\alpha)|=2$. Let $B(X)=\{\alpha\in A(X):l'(I(\alpha))\geq l'(\overline{I}(\alpha))\}$.
For the first part, it suffices to show that $\mathcal{B}=\{B(Z):Z\in\mathcal{R}_{3}\}$
has an SDR $g$: for each $X\in\mathcal{R}_{3}$ set $I_{X}=I(\alpha)$, and $f(v_{I_{X}})=\alpha$,
where $\alpha=g(B(X))$. 

By \eqref{R1}, $N_{3}(X)\subseteq C_{1}\subseteq C_{2}$; so $n_{3}(X)\leq u_{2}+r_{1}$.
By \eqref{spb} 
\begin{equation}
n_{2}(X)+2n_{3}(X)\geq4l-|W(X)|\geq4(l-k)+1\geq2k-1.\label{n2+2n3-eq}
\end{equation}
Thus 
\begin{eqnarray}
|A(X)| & = & n_{2}(X)+n_{3}(X)-|C_{2}|\geq n_{2}(X)+2n_{3}(X)-n_{3}(X)-|C_{2}|\label{A(X)-eq:}\\
 & \geq & 2k-1-(2u_{2}+u_{3}+2r_{1})\geq2r_{3}-1.\nonumber 
\end{eqnarray}
If $\alpha\in A(X)\smallsetminus B(X)$ then $A(X)\cap L(\overline{I}(\alpha))\subseteq B(X)$.
So $|B(X)|\geq\lceil|A(X)|/2\rceil\geq r_{3}.$ Hence $\mathcal B$ 
has an SDR $g$. 

Now suppose $\dot{Y}$ is defined in Step 11(b). Then $b=1$, $u_{1}=r_{2}=0$, and
$|W(\dot{Y})|\leq3k-1-u_{2}-r_{1}$. As $b=1$, $k$ is odd; so $k\ge3$. If $r_{3}\geq2$ then fix $Z\in\mathcal{R}_{3}-\dot{Y}$.
A partition $\mathcal{Q}=\{I$,$\overline{I}\}$ of $\dot{Y}$ into pairs is \emph{bad}
if $l'(I)=0$ or $l'(\overline{I})=0$; else it is good. It is \emph{weak} if $r_{3}\geq2$,
$L'(I)\cup L'(\overline{I})\subseteq B(Z)$ and $|B(Z)|=r_{3}$; else it is strong. 

For the second part, it suffices to show that $\dot{Y}$ has a good, strong partition: If
$\{\dot{I},\overline{\dot{I}}\}$ is a  good, strong partition then choose 
$\alpha,\beta\in L'(I)\cup L'(\overline{I})$ with $|B(Z)-\alpha|\geq r_{3}$
and $\alpha\in L'(I)\mbox{ iff }\beta\in L'(\overline{I})$. Then $\alpha$ and $\beta$
are the representatives for $L'(I)$ and $L'(\overline{I})$, or \emph{vice versa}.
We are done if  $r_3 =1$. If $r_{3}\geq2$ then continue by greedily choosing an SDR of $\mathcal{B}-B(\dot{Y)}-B(Z)+(B(Z)-\alpha)+L'(I)+L'(\overline{I})$
by picking representatives for $\mathcal{B}-B(\dot{Y})-B(Z)$, and finally
picking a representative for $B(Z)-\alpha$. 

Using the first half of \eqref{n2+2n3-eq}, 
\[
n_{2}(\dot{Y})+2n_{3}(\dot{Y})\geq4l-|W(\dot{Y})|\geq2l+u_{2}+r_{1}.
\]
So by \eqref{A(X)-eq:}, 
\[
|A(\dot{Y})|\geq2l+u_{2}+r_{1}-(2u_{2}+u_{3}+2r_{1})\geq2l-u_{2}-u_{3}-r_{1}\geq2k+r_{3}-1.
\]

First suppose for a contradiction that $\dot{Y}$ has no good partition. For 
each partition $\mathcal P$ of $X$ into pairs, choose $I\in \mathcal P$ with  $L(I)\cap A(\dot Y)=\emptyset$. Using Claim~\ref{cl}, there
exists $w\in\dot{Y}$ such that either (i) $L(wx)\cap A(\dot{Y})=\emptyset$ for
all $x\in\dot{Y}\smallsetminus w$ or (ii) $L(xy)\cap A(\dot{Y})=\emptyset$ for
all $xy\subseteq\dot{Y}\smallsetminus w$. If (i) holds then $L(w)\cap A(\dot Y)=\emptyset$. This yields the contradiction
\[
l+2k+r_{3}-1\leq l(w)+|A(\dot{Y})|\leq|W(\dot{Y})|\leq3k-1-u_{2}-r_{1}<l+2k-1.
\]
If (ii) holds then $A(\dot Y)\subseteq L(w)$, and so 
 $l<|A(\dot{Y})|\leq l(w)$, another contradiction.

So $\dot{Y}$ has a good partition (say) $\mathcal{Q}_{1}=\{xy,zw\}$. 
Suppose $\mathcal{Q}_{1}$ is weak. Then   $r_3\ge 2$ and $|A_{0}|\geq2k-1$, where  $A_{0}:=A(\dot{Y})\smallsetminus B(Z)\subseteq A(\dot{Y})\smallsetminus (L'(xy)\cup L'(zw))$. The former implies $2\le r_3\le l-k\le2k-l$;  so (*)  $l\le2k-2$. If the other two partitions of $\dot{Y}$ are both bad then
there is $v\in\dot{Y}$ with $A_{0}\subseteq L(v)$. So $2k-1\le|A_0|\leq l$
contradicting (*). 
 Say $\mathcal{Q}_{2}=\{xw,yz\}$ is good. If $\mathcal{Q}_{2}$
is weak then $A_{0}\subseteq A(\dot{Y})\smallsetminus (L'(xy)\cup L'(zw)\cup L'(xw)\cup L'(yz))$. Then $|L'(xz)\cup L'(yw)|\geq2k-1$. 
So $\mathcal{Q}_{3}=\{xz,yw\}$ is strong. By (*), $l'(xz),l'(yw)\leq l<2k-1$. Thus $l'(xz),l'(yw)\ge1$, and so $\mathcal{Q}_{3}$
is also good. 
\end{proof}
\begin{lem}
\label{R2-L}For each $X\in\mathcal{R}_{2}$ there is a pair $I_{X}\subseteq X$
such that $\{L(I_{X}):X\in\mathcal{P}\smallsetminus\mathcal{U}_{4}\}\cup\{L(\overline{I}_{X}):X\in\mathcal{U}_{1}\cup\mathcal{R}_{2}\}$
has an SDR $f_{4}$ that extends $f_{3}$.\end{lem}
\begin{proof}
 Each $X\in\mathcal{U}_{1}$ satisfies $L(I_{X}),L(\overline{I}_{X})\ge k$. Thus
$|L(I_{X})\smallsetminus C_{3}|,|L(\overline{I}_{X})\smallsetminus C_{3}|\geq k-u_{2}-u_{3}-r_{1}-r_{3}\ge u_{1}$.
By Theorem~\ref{thm:ERT}, $\{L(I_{X})\smallsetminus C_{3},L(\overline{I}_{X})\smallsetminus C_{3}:X\in\mathcal{U}_{1}\}$
has an SDR, and so $f_{3}$ can be extended to an SDR $g$ for $\mathcal{L}(M_{3}'),$
where $M_{3}':=M_{3}\cup\{I_{X},\overline{I}_{X}:X\in\mathcal{U}_{1}\}$. Let $C^{g}=\ra(g)$.
Then $|C^{g}|=2u_{1}+u_{2}+u_{3}+r_{1}+r_{3}$. Consider any $X=xyzw\in\mathcal{R}_{2}$.
Let $A(X)=N_{2}(X)\smallsetminus C^{g}$. Again by Theorem~\ref{thm:ERT} it suffices
to show:
\begin{equation}
(\exists I_{X}\subseteq X)[|I_{X}|=2\wedge|L(I_{X})\cap A(X)|\ge r_{2}\wedge|L(\overline{I}_{X})\cap A(X)|\ge r_{2}].\label{R2-LAim}
\end{equation}

Observe $\sigma_{2}(X)=n_{2}(X)+3n_{3}(X)$ and $\sigma_{3}(X)=n_{3}(X)$. So $n(X)=n_{2}(X)+n_{3}(X)=\sigma_{2}(X)-2\sigma_{3}(X)$.
By \eqref{R1}, $N_{3}(X)\subseteq C^{g}$, and by \eqref{U3} $\sigma_{3}(X)\le u_{2}+r_{1}$.
So
\begin{eqnarray}
n(X) & = & \sigma_{2}(X)-2\sigma_{3}(X)\geq5(l-k)+2u_{1}+2u_{2}+u_{3}+r_{1}+r_{2}-(u_{2}+r_{1})\nonumber \\
 & \geq & 5(l-k)+2u_{1}+u_{2}+u_{3}+r_{2}\mbox{ and}\label{duplicate}\\
|A(X)| & = & |N_{2}(X)\smallsetminus C^{g}|=|N_{2}(X)\cup N_{3}(X)\smallsetminus C^{g}|=n(X)-|C^{g}|\nonumber \\
 & \geq & 5(l-k)+2u_{1}+u_{2}+u_{3}+r_{2}-(2u_{1}+u_{2}+u_{3}+r_{1}+r_{3})\nonumber \\
 & \geq & 5(l-k)-r_{1}+r_{2}-r_{3}\geq4(l-k)+2r_{2}.\label{|A|}
\end{eqnarray}

Suppose \eqref{R2-LAim} fails. Then for each of the three partitions of $X$ into
pairs, there is a pair $uv$ with $|L(uv)\cap A(X)|\leq r_{2}-1$. 
Using Claim~\ref{cl}, there exists $v\in X$ such that either (i) $|L(vw)\cap A(X)|\leq r_{2}-1$ for all $w\in X-v$ or (ii) $|L(vw)\cap A(X)|\leq r_{2}-1$ for all $w\in X-v$. 

If (i) holds then
\[
|L(v)\cap N(X)|\leq|C^{g}|+\sum_{w\in X-v}|L(vw)\cap A(X)|\leq|C^{g}|+3r_{2}-3.
\]
Since $|L(w)\cap N(X)|\leq l$ for all $w\in X-v$,
\[2n(X)\le\sum_{v\in X}|L(v)\cap N(X)| \le 3l+(|C^{g}|+3r_{2}-3). \]
Using $|C^{g}|=2u_{1}+u_{2}+u_{3}+r_{1}+r_{3}$ and 
 \eqref{duplicate} implies 
\begin{eqnarray}
 10(l-k)+4u_{1}+2u_{2}+2u_{3}+2r_{2} &\le& 3l-2u_{1}+u_{2}+u_{3}+r_{1}+r_{3}+3r_2-3
\label{case1} \\
  4l-k+(6l-9k+3)+2u_{1}+u_{2}+u_{3}&\le& 3l+r_{1}+r_{2}+r_{3}
  \le 4l-k. \nonumber
\end{eqnarray}
Since $6l-9k=-3b$, both $b=1$ and $0=u_{1}=u_{2}=u_{3}$. Now, by \eqref{U3}, $\mu_{2}(X)\leq l-k$.
So $|L(w)\cap N(X)|\leq3(l-k)$ for all $w\in X$. Strengthening the estimate in \eqref{case1} yields
the contradiction:
\begin{eqnarray*}
 10(l-k)+2r_{2}&\le& 9(l-k)+(|C^{g}|+3r_{2}-3)\\
 l-k &\le& r_{1}+r_{2}+r_{3}-3< l-k.
\end{eqnarray*}

Thus (ii) holds.  So  
\begin{eqnarray}
|A(X)|\le l(v)+\sum_{wx\subseteq X-v}|L(uv)\cap A(X)|\leq l+3(r_{2}-1)\label{case 2}.
\end{eqnarray}
Using    \eqref{|A|}, \eqref{case 2} and $2l-3k=-b$, this yields the contradiction 
\begin{eqnarray*}
4(l-k)+2r_{2}\le |A(X)| & \le &  l+3(r_{2}-1)\\
l-k+2\le3l-4k+3 &\leq&  r_{2}\leq l-k.
\end{eqnarray*}
\end{proof}
\begin{lem}
\label{HC-lem}Suppose $X=xyzw\in\mathcal{U}_{4}$ and $|W(X)|<|G'|$. Then $b=1$,
$u_{1}=0=r_{3}$, $u_{2}+u_{3}\geq1$, and there exists a pair $J\subseteq X$
such that:
\begin{enumerate}
\item \textup{$L(J)\nsubseteq C_{4};$}
\item $|W(\{v_{J},v\})|\geq2k-1$ and if\textup{ $|W(\{v_{J},v\})|=2k-1$ then} \textup{$|W(\{v_{J},v\}\cup \dot Z)\cup C_{4}|\geq2k$}
for both $v\in\overline{J}$;
\item $|W(\overline{J}+v_{J})|\geq|G'|-1$; in particular $|W(X)|\ge|G'|-1$.
\end{enumerate}
\end{lem}
\begin{proof}
Now $|G'|=3k-u_{1}-u_{2}+u_{4}-r_{1}-r_{2}$. Observe that 
\begin{equation}
\sigma_{2}(X)-\sigma_{3}(X)\geq5(l-k)+2u_{1}+2u_{2}+u_{3}+r_{1}+r_{2}+1,\label{sigma2-3:eq}
\end{equation}
since otherwise inclusion-exclusion yields the contradiction: 
\begin{eqnarray*}
|W(X)| & = & \sigma_{1}(X)-\sigma_{2}(X)+\sigma_{3}(X)\\
 & \geq & 4l-5(l-k)-2u_{1}-2u_{2}-u_{3}-r_{1}-r_{2}\\
 & \geq & 3k+(2k-l-u_{1}-u_{2}-u_{3})-u_{1}-u_{2}-r_{1}-r_{2}\\
 & \geq & 3k-u_{1}-u_{2}+u_{4}-r_{1}-r_{2}=|G'|>|W(X)|.
\end{eqnarray*}

By \eqref{sigma2-3:eq} and \eqref{R2}, $r_{1}+r_{2}=l-k$ and $r_{3}=0$. Consider
any pair $I=xy\subseteq X$. Then 
\begin{eqnarray}
\,\,\,\,\,\,\,\,\,\,\,\,\,|W(\overline{I}+v_{I})| & \geq & l(xy)+l(z)+l(w)-l(xyz)-l(xyw)-l(zw)\label{l(I)cupIbar}\\
 & \geq & 2l-2u_{2}+\Delta_{1}(I)+\Delta_{2}(I)\nonumber \\
|G'|-|W(\overline{I}+v_{I})| & \leq & b-2u_{1}+(u_{1}+u_{2}+u_{4}-l+k)-\Delta_{1}(I)-\Delta_{2}(I)\\
1 & \leq & 2b-2u_{1}-u_{3}-\Delta_{1}(I)-\Delta_{2}(I).\label{G-W}
\end{eqnarray}
By \eqref{G-W}, $\Delta_{1}(I)+\Delta_{2}(I)\leq1$. As $\Delta_{1}(I)=-\Delta_{1}(\overline{I})$
and $\Delta_{2}(I),\Delta_{2}(\overline{I})\geq0$, we could choose $I$ with $\Delta_{1}(I)+\Delta_{2}(I)\geq0$.
So $b=1$, $u_{1}=0$, $u_{3}\leq1$, and 
\begin{equation}
1\leq|G|-|W(\overline{I}+v_{I})|\leq2-u_{3}-\Delta_{1}(I)-\Delta_{2}(I)\leq2.\label{1b:eq}
\end{equation}
Furthermore, using $\Delta_{1}(I)=-\Delta_{1}(\overline{I})$ again, 
\begin{equation}
0\leq4u_{2}-\sigma_{3}(X)=\Delta_{2}(I)+\Delta_{2}(\overline{I})=\Delta_{1}(I)+\Delta_{2}(I)+\Delta_{1}(\overline{I})+\Delta_{2}(\overline{I})\leq2.\label{u3s3}
\end{equation}

By \eqref{sigma2-3:eq}, $r_{1}+r_{2}=l-k$,  $\sigma_{2}(X)\leq6\mu_{2}(X)$,
 \eqref{U3}, and $\sigma_{3}=4u_{2}-\Delta_{2}(I)-\Delta_{2}(\overline{I})$, 
\begin{eqnarray}
1+6(l-k)+2u_{2}+u_{3}+\sigma_{3}(X) & \leq & \sigma_{2}(X)\leq6(l-k+u_{2}+u_{3})\label{s}\\
1+u_{3}+6(l-k+u_{2})-\Delta_{2}(I)-\Delta_{2}(\overline{I}) & \leq & \sigma_{2}(X)\leq6(l-k+u_{2}+u_{3}).\label{sigma2}
\end{eqnarray}
By \eqref{s} $u_{2}+u_{3}\geq1$. So the first three assertions of the lemma have
been proved. It remains to find a pair $J\subseteq X$ satisfying (1--3). 

First suppose $u_{3}=1$. By \eqref{1b:eq}, $\Delta_{1}(I)+\Delta_{2}(I)=0$ for
all pairs $I\subseteq X$. So $\Delta_{1}(I)\leq0$ and $\Delta_{1}(\overline{I})\leq0$.
As $\Delta_{1}(I)=-\Delta_{1}(\overline{I})$, this implies $\Delta_{1}(I)=0=\Delta_{1}(\overline{I})$.
So $\Delta_{2}(I)=0=\Delta_{2}(\overline{I})$. By \eqref{sigma2}, there exists a
pair $I\subseteq X$ with $l(I)=l-k+u_{2}+u_{3}$. As $\Delta_{1}(I)=0$, $l(\overline{I})=l-k+u_{2}+u_{3}$.
Thus 
\[
|W(\{v_{I},v_{\overline{I}}\})|=l(I)+l(\overline{I})=2(l-k+u_{2}+u_{3})>2(l-k)+u_{2}+u_{3}\geq|C_{4}|.
\]
Pick $J\in\{I,\overline{I}\}$ such that $L(J)\nsubseteq C_{4}$. Then (1) holds.
For (2), let $v'\in\overline{J}$, and observe 
\[
|W(\{v_{J},v'\})|=l(J)+l(v')-l(J+v')\geq2l-k+u_{2}+u_{3}-u_{2}=2k.
\]
Thus $(2)$ holds. As $u_{3}=1$, \eqref{1b:eq} implies (3).

Otherwise $u_{3}=0$. Then $u_{2}\geq1$, and so $\dot Z$ is defined in Step 3. 
Put $C_{0}:=C_{4}\cup W(\dot Z')$. By Step~3, $|C_{0}|\geq|W(\dot Z')|\geq l+u_{2}$.  Call
a vertex $x\in X$ \emph{bad} if $|L(x)\cup C_{0}|\leq2k-1$; otherwise $x$ is \emph{good}.
If $x$ is bad then $|C_{0}\smallsetminus L(x)|\le 2k-1-l\leq l-k$. If another vertex $y$
is also bad, then using \eqref{U3} and \eqref{1b:eq}, 
\begin{eqnarray*}
l-k+u_{2}\geq l(xy) & \geq & |L(xy)\cap C_{0}|\geq|C_{0}|-|C_{0}\smallsetminus L(x)|-|C_{0}\smallsetminus L(y)|\\
 & \geq & l+u_{2}-2(l-k)\geq l-k+u_{2}+1,
\end{eqnarray*}
a contradiction. So at most one vertex of $X$ is bad. 

Call a pair $I\subseteq X$ \emph{bad} if $L(I)\subseteq C_{4}$; otherwise $I$
is \emph{good}. Note that if $I$ is good then $I$ satisfies (1).  By \eqref{u3s3},
\eqref{sigma2}, and $u_{3}=0$, $6(l-k+u_{2})-1\leq\sigma_{2}\leq6(l-k+u_{2})$; 
and so by \eqref{U3}, $ $every pair $I\subseteq X$ satisfies 
\[
l-k+u_{2}-1\leq l(I)\leq l-k+u_{2}.
\]
If the upper bound is sharp then call $I$ \emph{normal}; otherwise call $I$\emph{
abnormal}. Then there is at most one abnormal pair. If $I$ is normal then $l(\overline I)\le l(I)$; so $\Delta_{1}(I)\geq0$.

By \eqref{U2}, every triple $T\subseteq X$ satisfies $l(T)\leq u_{2}$. If equality
holds then call $T$ \emph{normal}; otherwise call $T$ \emph{abnormal}; if $|L(T)\cap C_{0}|\leq u_{2}-2$
then call T \emph{very abnormal}.  Suppose two  pairs $I,J\subseteq T$ are both
bad. At least one, say $I$, is normal.  Then
\begin{eqnarray}
2(l-k)+u_{2} & \geq & |C_{4}|\geq|L(I)\cup L(J)|\geq l-k+u_{2}+l(J)-l(I\cup J)\label{2bad-eq}\\
l(I\cup J) & \geq & l(J)-l+k =
\begin{cases}
u_2 & \mbox{if $J$ is normal}\\
u_2-1 & \mbox{if $J$ is abnormal}
\end{cases}.\nonumber 
\end{eqnarray}
So an abnormal triple contains at most one bad, normal pair, and a very abnormal
triple contains at most one bad pair. A pair $I$ contained in an abnormal triple
satisfies $\Delta_{2}(I)\geq1$. 

Let $J$ be a good, normal pair contained in a abnormal triple $T$ with $w\in X\smallsetminus T$.
Then  $\Delta_{1}(J)+\Delta_{2}(J)\geq1$. So $J$ satisfies
(3) by \eqref{1b:eq}. Also, 
\begin{equation}
|W(v_{J},v)|=l(J)+l(v)-l(J+v)\geq \begin{cases}
2l-k+u_{2}-(u_{2}-1)=2k & \mbox{if $v\in T\smallsetminus J$}\\
2l-k+u_{2}-u_{2}=2k-1 & \mbox{if $v=w$}
\end{cases}.\nonumber 
\label{A-eq}
\end{equation}
So ({*}) $J$ satisfies (2), provided $|W(v_{j},w)\cup C_{0}|\geq2k$. In particular, (2) holds if $w$ is good.

By \eqref{sigma2} and \eqref{u3s3}, $1\leq\Delta_{2}(I)+\Delta_{2}(\overline{I})\leq2$. As $\sigma_{3}=4u_{2}-\Delta_{2}(I)-\Delta_{2}(\overline{I})$, 
we have $4u_{2}-2\leq\sigma_{3}(X)=4u_{2}-1$. In the first case there is one abnormal
triple. In the second case, either there is a very abnormal triple or there are two
abnormal triples. 

First suppose there are two abnormal triples. Choose an abnormal triple $T$ so that if there is a
bad vertex then it is in $T$. As $T$ contains three pairs and at most one is bad
and at most one is abnormal, $T$ contains a good, normal pair $J$. Say $J=yz$,
$T=xyz$, and $w\in X\smallsetminus T$. Then $w$ is good, and thus $J$ satisfies
(2) by ({*}).

Otherwise, let $T=xyz$ be the only abnormal triple and $w\in X\smallsetminus T$.
There is at most one abnormal pair, and only if $T$ is very abnormal. So $T$ contains
at most one bad pair. Now suppose $T$ has two good, normal pairs $xy$ and $yz$.
By ({*}), some $J\in P:=\{xy,yz\}$ satisfies (2), unless $C_{0}\subseteq L(J)\cup L(w)$
for both $J\in P$. Then, using $u_1=u_3=r_3=0$,
\[
l+u_{2}=|C_{0}|\leq|L(xy)\cup L(w)|+|L(yz)\cup L(w)|-|L(xy)\cup L(yz)\cup L(w)|.
\]
As $T$ is abnormal, and both $xy$ and $yz$ are normal, 
\begin{eqnarray*}
|L(xy)\cup L(yz)\cup L(w)| & = & l(xy)+l(yz)+l(w)-l(xyw)-l(yzw)-l(xyz)\\
 & \geq & 3l-2k+2u_{2}-(3u_{2}-1)=k+l-u_{2}.
\end{eqnarray*}
Combining the last two expressions yields the contradiction,
\[
l+u_{2}\leq|C_{0}|\leq2(2k-1)-(k+l-u_{2})=3k-1-l+u_{2}-1=l+u_{2}-1.
\]

Otherwise, $T$ is very abnormal, and (say) both $xz$ is bad and $J=yz$ is normal. As $T$ contains at most one bad pair, $yz$ is also good. Since $xz$ is bad, $xz\subseteq C_0$. 
Now
\[
|C_0\smallsetminus W(\{v_J,w\})|\ge|L(xz)\smallsetminus(L(w)\cup L(J))|\geq l-k+u_{2}-(u_{2}-2)\geq1,
\]
and (2) holds by ({*}), since $\emptyset\ne L(xz)\smallsetminus(L(w)\cup L(J))\subseteq C_{0}$
implies 
\[
|W(\{v_{J},w\}\cup C_{0})|\geq|W(\{v_{J},v\})|+1\geq2k.
\]
\end{proof}
\begin{lem}
\label{HC'-lem}Suppose $b=1$ and $X=xyzw\in\mathcal{U}_{4}$. If 
\[
|W(xyz)|\leq2k+u_{4}-1<|W(X)|
\]
 then $u_{1}=0$ and there exists a pair $J\subseteq X$ such that:
\begin{enumerate}
\item \textup{$L(J)\nsubseteq C_{4};$}
\item $|W(\{v_{J},v\})|\geq2k$ for $v\in xyz\smallsetminus J$ and $|W(\{v_{J},w\})|\geq2k-1+u_{3}$;
and
\item $|W(\overline{J}+v_{J})|\geq2k+u_{4}$.
\end{enumerate}
\end{lem}
\begin{proof}
Consider a pair $vv'\subseteq xyz$. Then
\begin{eqnarray*}
2k+u_{4}-1 & \geq & |W(xyz)|\geq|W(vv')|\geq l(v)+l(v')-l(vv')\\
 & \geq & 2l-(l-k+u_{2}+u_{3})\geq3k-1-k+u_{1}+u_{4}\\ &\ge& 2k+u_1+u_4-1\ge2k.
\end{eqnarray*}
So $u_{1}=0$, $l(vv')=l-k+u_{2}+u_{3}$, and $W(xyz)=W(vv')$. Since $vv'$ is arbitrary,
every color in $W(xyz)$ appears in at least two of the lists $L(x),$ $L(y)$, $L(z)$. So $W(\{v_{J},v\})=W(xyz)$ and  $|W(\{v_{J},v\})|\geq2k$
for every pair  $J\subseteq xyz$ and vertex  $v\in xyz\smallsetminus J$.
 As $|C_{4}|<2k\le|W(xyz)|$,
there is a pair $J\subseteq xyz$ with $L(J)\nsubseteq C_{4}$.  Furthermore,
\[
|W(\{v_{J},w\})|\geq l(J)+l(w)-l(J+w)\geq l-k+u_{2}+u_{3}+l-u_{2}=2k-1+u_{3}.
\]
Finally, as $W(\{v_{J},v\})=W(xyz)$ for $v\in xyz\smallsetminus J$, 
\[
|W(\overline{J}+v_{J})|=|W(\{v_{J},v\})\cup W(w)|=|W(xyzw)|\geq2k+u_{4}.
\]
\end{proof}
\begin{lem}
$G'$ is $L$-choosable. \end{lem}
\begin{proof}
First observe that if $k$ is even then $b=\dot{u}=\ddot{u}=\dot{r}=0$ and $H=G'$.
In this case the following argument is much simpler. 

Using Hall's Theorem it suffices to show $|S|\leq|W|:=\left|\bigcup_{x\in S}L(x)\right|$
for every $S\subseteq V(H)$. Suppose for a contradiction that $|S|>|W|$ for some
$S\subseteq V(H)$. We consider several cases. Each case assumes the previous cases
fail.
\begin{enumerate}[label=\textbf{Case \arabic*:}, wide=\parindent, labelindent=0pt ]
\item There is $X\in\mathcal{U}_{4}$ with $|S\cap X|=4$. 
 Then $|W|<|S|\le|G'|$. 
By Lemma~\ref{HC-lem}, $b=1$ and $|G'|-1=|W(X)|<|S|=|G'|$. So
\ref{extra-Step}(a) is executed, and $S=V(G')$. In particular, $\dot X\subseteq S$. Thus 
\[
|S|\leq|H|=|G'|-1\leq|W(\overline{J}+v_{J})|\leq|W(\dot X)|\leq|W|.
\]

\item  There exists $Z=xyzw\in\mathcal{U}_{3}$ with $|S\cap Z'|=3$. Now $|S|\leq3k-u_{1}-u_{2}-r_{1}-r_{2}-\dot{r}$, since Case 1 fails.
Say $I_{Z}=xy$. By Step 4, $\Delta_{1}(xy)\ge0$ and $l(xyz)+l(xyw)=2u_2-\Delta_2(xy)$.   By Step 3, $l(xyz)+l(xyw)\leq u_{2}+r_{1}$. So
\begin{eqnarray}\label{case2}
|W| & \geq & |W(Z')|\ge l(xy)+l(z)+l(w)-l(xyz)-l(xyw)-l(zw) \\
 & = & 2l+\Delta_{1}(xy)-2u_{2}+\Delta_2(xy)=3k-b+\Delta_{1}(xy)-2u_{2}+\Delta_2(xy)
 \notag\\&\geq& 3k-b+\Delta_{1}(xy)-u_{2}-r_{1}\ge |S|-b.\notag
\end{eqnarray}
As $|S|>|W|$ equality holds throughout. 
Thus $b=1$, $u_{1}=r_{2}=\dot{r}=\Delta_{1}(xy)=0$,  $r_{1}\leq u_{2}$, and (*) 
$Y'\subseteq S$ for all $Y\in \mathcal R_3$.
If $u_{4}=0$ then 
\[
k=l-k+u_{2}+u_{3}\leq l(xy)=l(\overline{xy})+\Delta_{1}(xy)=l(\overline{xy}).
\]
By \eqref{U-1-eq} this contradicts $u_{1}=0$. So $u_{4}\geq1$, $u_{3}+u_{4}\geq2$,
and 
\[
r_{1}+r_2\leq u_{2}+0=2k-l-u_{3}-u_{4}\leq l-k-1.
\]
Thus $r_3\ge1$. Say $Y\in\mathcal R_3$.  By (*), $Y'\subseteq S$; by \eqref{case2}, $|W(Y')|\le |W|=3k-1-u_2-r_1$.
 So, using  $b=1$ and $u_1=0=r_2$, Step 11(b) is executed, and $\dot r=1$, a contradiction.

\item There exists $X=wxyz\in\mathcal{R}_{3}$ with $|S\cap X'|=3$. Say $I_{X}=xy$. Now
$|S|\leq3k-u_{1}-u_{2}-u_{3}-r_{1}-r_{2}-\dot{r}$. By Step~7, $l'(xy)\geq l'(wz)$.
By \eqref{R1}, 
 $N_3(X)\subseteq C_1\subseteq C_2$. So 
$l'(xyz)=0=l'(xyw)$. Set $t=|C_{2}\cap W|$. Then $t\leq u_{2}+u_{3}+r_{1}$.
So 
\begin{eqnarray*}
|W| & = & |W\smallsetminus C_{2}|+|C_{2}\cap W|\geq l'(xy)+l'(z)+l'(w)-l'(xyz)-l'(xyw)-l'(zw)+t\\
 & \ge & l'(xy)+l(z)-t+l(w)-t-l'(zw)+t\\
 & \geq & 3k-b-(u_{2}+u_{3}+r_{1})\ge |S|-b.
\end{eqnarray*}
Thus $b=1$, $0=r_{2}=u_{1}=\dot r$, and $|W(X)|\le |W|\le3k-1-u_2-r_1$. So Step 11(b) is executed, and $\dot{r}=1$, a contradiction.

\item There exists $X\in\mathcal{U}_{4}$ with $|S\cap X'|=3$. As the previous cases fail,
$|S|\leq2k+u_{4}$. Let $xy\subseteq S\cap X'\smallsetminus M$. By \eqref{U3},  
\begin{eqnarray*}
|W| & \geq & l(x)+l(y)-l(xy)\geq3k-b-(l-k+u_{2}+u_{3})\\
 & \geq & 2k+(2k-l)-(u_{2}+u_{3})-b\geq2k+u_{1}+u_{4}-b\ge |S|-b.
\end{eqnarray*}
So $b=1$, $u_{1}=0$,  $|W|=2k+u_{4}-1$, and $|S|=2k+u_4$. Thus $S$ has exactly two vertices
in every class of $\mathcal{P}'\smallsetminus\mathcal{U}'_{4}$ and exactly three
vertices in every class of $\mathcal{U}'_{4}$. In particular, $\dot{Z}'\subseteq S$.
If $\dot{u}=1$, 
 then $\dot{X}'\subseteq S$ and 
$|W(\dot X')|\geq|G'|-1\ge2k+u_4\ge |S| $ by
Lemma~\ref{HC-lem}; else $|W(X)|\geq2k+u_{4}$. If $\ddot{u}=1$ then $\ddot{X}'\subseteq S$
and $|W|\geq|W(\ddot{X}')\geq|S|$ by Lemma \ref{HC'-lem}; else $X=X'$. As Step 11(c) is not executed,
\[
|W|\geq|W((S\cap X)\cup \dot Z)|\ge 2k+u_4\geq|S|.
\]

\item There exists $X\in\mathcal{U}_{1}$ with $|S\cap X'|=2$. Say $S\cap X'=\{v_{I},v_{\overline{I}}\}$.
As the previous cases fail, $|S|\leq2k$. Now 
\begin{eqnarray*}
|W|  \geq  L(v_{I})+L(v_{\overline{I}})\geq2k\geq|S|.
\end{eqnarray*}

\item There exists $X\in\mathcal{U}_{3}$ with $|S\cap X'|=2$. Say $S\cap X'=vv'$. As
the previous cases fail, $|S|\leq2k-u_{1}$. If $v,v'\notin M$ then $\overline{I}_{X}=vv'$.
By \eqref{U-1-eq}, $l(\overline{I}_{X})\leq k-1$. So
\[
|W(vv')|\ge l(v)+l(v')-l(vv')\ge2l-(k-1)\geq2k\ge|S|.
\]
Otherwise $v=v_{xy}$, where $I_{X}=xy,$ and $v'=z\notin M$. Then
\begin{eqnarray*}
|W(vv')| & \geq & l(v_{xy})+l(z)-l(xy+z)\\
 & \geq & l-k+u_{2}+u_{3}+l-u_{2}\geq2k-b+u_{3}\geq2k\geq|S|.
\end{eqnarray*}

\item There exists $X\in\mathcal{U}_{4}$ with $|S\cap X'|=2$. Say $S\cap X'=vv'$. If
possible, choose $X$ so that $S\cap X'\cap M=\emptyset$. As the previous cases
fail, $|S|\leq2k-u_{1}-u_{3}$. If $v,v'\notin M$ then 
\begin{eqnarray}
|W(vv')| & = & l(v)+l(v')-l(vv')\geq2l-(l-k+u_{2}+u_{3})\notag\label{2U4-eq}\\
 & \geq & 2k-b+u_1+u_{4}\geq2k\geq|S|.
\end{eqnarray}
Else $b=1$, and (say) $v\in M$. By Step 11, $v=v_{\dot{I}}$ or $v=v_{\ddot{I}}$, and
$u_{1}=0$.

If $v=v_{\dot{I}}$ then Step 11(a) was executed. So (i) $r_{3}=0$, (ii) $|W(vv')|\geq2k-1$,
and (iii) if $|W(vv')|=2k-1$ then $u_{2}\geq1$ and $|W(vv'\cup\dot{Z}')\cup C_{4}|\geq2k$.
Since 
\[
2k\geq|S|>|W|\geq|W(vv')|\geq2k-1,
\]
$|S|=2k$. Thus $S$ contains exactly two vertices of each part $Y'\in\mathcal{P}'$. In particular, $\dot{Z}'\subseteq S$.
The choice of $X$ implies $u_{3}=0$ and $u_4=1$; thus $u_{2}=l-k\geq1$. Since $u_{3}=0=r_{3}$, $M_{4}\subseteq S$. So $|W|\geq|W(vv'\cup\dot{Z}')\cup C_{4}|\geq2k$,
a contradiction.

Otherwise $x=v_{\ddot{I}}$. Then Step 11(c) was executed. So  there
is a part $\ddot{X}=xyzw\in\mathcal{U}_{4}$ with $\ddot{I}=xy$ such that 
\[
|W(xyz\cup \dot Z)|\leq2k+u_{4}-1<|W(\ddot{X})|,
\]
 $|W(\{v_{xy},w\})|\geq2k-1+u_{3}$, and $|W(\{v_{xy},z\})|\geq2k$. 
 So
we are done, unless $v'=w$ and 
\[
2k\geq|S|>|W(\{v_{xy},w\})|\geq2k-1+u_{3}.
\]
 Thus $u_3=0$ and $|S|=2k$. So $S$ contains exactly two vertices
of each class $Y'\in\mathcal{P}'$. In particular, $\dot{Z}'\subseteq S$. As
$|W(\ddot X)|>|W(xyz\cup \dot Z)|$,  we have
 $|L(w)\smallsetminus W(xyz\cup\dot{Z}')|\ge 1$. So
\[
|W|\geq|W(\{v_{xy},w\}\cup \dot{Z}')|\geq W(\dot{Z}')+1=l+u_{2}+1=2l-k+1=2k.
\]

\item There exists $X=xyzw\in\mathcal{U}_{2}$ with $|S\cap X'|=2$. Say $S\cap X'=\{v_{I},w\}$.
As the previous cases fail, $|S|\le2k-u_{1}-u_{3}-u_{4}=l+u_{2}$. Since $L(xyz)\cap L(w)=\emptyset$,
we have 
\[
|W|\geq|W(X')|\ge l(xyz)+l(w)\ge u_{2}+l\ge|S|.
\]

\item Otherwise. As the previous cases fail, 
\[
|S|\leq u_{1}+u_{2}+u_{3}+u_{4}+2|\mathcal{R}|=l.
\]
As $\mathcal{L}(M)$ has an SDR, there is a vertex $x\in S\smallsetminus M$. Thus
$|W|\geq l(x)=l\geq|S|$. 
\end{enumerate}
\end{proof}
We are done.
\end{proof}

\section{On-line Choosability}

By Theorem~\ref{M-th}, $\ch(K_{4*3})=4$. Using a computer we have checked that $\ch^{OL}(K_{4*3})=5$, but do not have a readable argument to verify the upper bound.  Here we prove the lower bound.

\begin{thm} 
$\ch^{OL}(K_{4*3}) \ge5$.
\end{thm}

\begin{figure}
\begin{tikzpicture}
 

\fill[fill=gray!30!white] (-0.8,5) -- (0.8,5) -- (0.8,6) -- (-0.8,6) -- (-0.8,5);
\fill[fill=gray!30!white] (1.7,6) -- (3.3,6) -- (3.3,6.55) -- (2.77,6.55) -- (2.77,7.1) -- (1.7,7.1) -- (1.7,5);
\fill[fill=gray!30!white] (4.2,6) -- (5.27,6) -- (5.27,6.55) -- (5.8,6.55) -- (5.8,7.1) -- (4.2,7.1) -- (4.2,5);
\fill[fill=gray!30!white] (6.75,6) -- (7.9,6) -- (7.9,7.1) -- (6.75,7.1) -- (6.75,6);
\fill[fill=gray!30!white] (1.7,3) -- (2.23,3) -- (2.23,3.55) -- (3.3,3.55) -- (3.3,4.1) -- (1.7,4.1) -- (1.7,3);
\fill[fill=gray!30!white] (4.2,3) -- (4.73,3) -- (4.73,4.1) -- (4.2,4.1) -- (4.2,3);
\fill[fill=gray!30!white] (4.73,3) -- (5.27,3) -- (5.27,3.55) -- (4.73,3.55) -- (4.73,3);
\fill[fill=gray!30!white] (6.7,3) -- (7.23,3) -- (7.23,3.55) -- (8.3,3.55) -- (8.3,4.1) -- (6.7,4.1) -- (6.7,3);
\fill[fill=gray!30!white] (1.7,0.5) -- (3.3,0.5) -- (3.3,1.05) -- (1.7,1.05) -- (1.7,0.5);
\fill[fill=gray!30!white] (4.2,0.5) -- (5.8,0.5) -- (5.8,1.05) -- (4.2,1.05) -- (4.2,0.5);


\node (a0) at (0,6) {
$\begin{array}{ccc}
4 & 4 & 4 \\
4 & 4 & 4 \\
4 & 4 & 4 \\
4 & 4 & 4
\end{array}$
};

\node (a1) at (2.5,6) {
$\begin{array}{ccc}
4 & 4 & 4 \\
4 & 4 & 4 \\
3 & 3 & \\
3 & 3 & \\
\end{array}$
};

\node (a2) at (5,6) {
$\begin{array}{ccc}
3 & 3 & 4 \\
3 & 3 & \\
3 & 3 & \\
3 & 3 & \\
\end{array}$
};

\node (a3) at (7.5,6) {
$\begin{array}{ccc}
3 & 3 & \\
3 & 3 & \\
2 & 2 & \\
2 & 2 & \\
\end{array}$
};

\node (a4) at (9.75,6) {
$\begin{array}{ccc}
2 & 2 \\
2 & 2 \\
2 & \\
2 &
\end{array}$
};

\node (b1) at (2.5,3) {
$\begin{array}{ccc}
3 & 3 & 4 \\
3 & 3 & 3 \\
3 & & \\
3 & & \\
\end{array}$
};

\node (b2) at (5,3) {
$\begin{array}{ccc}
3 & 3 & 3 \\
3 & 3 & \\
2 &  & \\
2 &  & \\
\end{array}$
};

\node (b3) at (7.5,3) {
$\begin{array}{ccc}
2 & 3 & 3 \\
2 & & \\
2 & & \\
2 & & \\
\end{array}$
};

\node (b4) at (9.75,3) {
$\begin{array}{cc}
2 & 2 \\
2 & \\
1 & \\
1
\end{array}$
};

\node (c1) at (2.5,0.5) {
$\begin{array}{ccc}
3 & 3 & 3 \\
2 & 3 & 3 \\
\end{array}$
};

\node (c2) at (5,0.5) {
$\begin{array}{ccc}
2 & 3 & 3 \\
2 & 2 & \\
\end{array}$
};

\node (c3) at (7.5,0.5) {
$\begin{array}{ccc}
2 & 2 & 2 \\
2 & & \\
\end{array}$
};

\node (d1) at (2.5,-1.5) {
$\begin{array}{ccc}
3 & 3 & 2 \\
2 & 2 & \\
\end{array}$
};

\node (d2) at (5,-1.5) {
$\begin{array}{cc}
2 & 2 \\
1 & 2
\end{array}$
};


\draw[->] (a0) to (a1);
\draw[->] (a1) to (a2);
\draw[->] (a2) to (a3);
\draw[->] (a3) to (a4);
\draw[->] (a1) to (b1);
\draw[->] (a2) to (b2);
\draw[->] (b1) to (b2);
\draw[->] (b2) to (b3);
\draw[->] (b3) to (b4);
\draw[->] (b1) to (c1);
\draw[->] (b2) to (c2);
\draw[->] (b3) to (c3);
\draw[->] (c1) to (c2);
\draw[->] (c2) to (c3);
\draw[->] (c1) to (d1);
\draw[->] (c2) to (d2);


\draw (-0.8,5) -- (0.8,5) -- (0.8,7.1) -- (-0.8,7.1) -- (-0.8,5);
\draw (1.7,5) -- (3.3,5) -- (3.3,7.1) -- (1.7,7.1) -- (1.7,5);
\draw (4.2,5) -- (5.8,5) -- (5.8,7.1) -- (4.2,7.1) -- (4.2,5);
\draw (6.75,5) -- (7.9,5) -- (7.9,7.1) -- (6.75,7.1) -- (6.75,5);
\draw (9.2,5) -- (10.3,5) -- (10.3,7.1) -- (9.2,7.1) -- (9.2,5);

\draw (1.7,2) -- (3.3,2) -- (3.3,4.1) -- (1.7,4.1) -- (1.7,2);
\draw (4.2,2) -- (5.8,2) -- (5.8,4.1) -- (4.2,4.1) -- (4.2,2);
\draw (6.7,2) -- (8.3,2) -- (8.3,4.1) -- (6.7,4.1) -- (6.7,2);
\draw (9.2,2) -- (10.3,2) -- (10.3,4.1) -- (9.2,4.1) -- (9.2,2);

\draw (1.7,0) -- (3.3,0) -- (3.3,1.05) -- (1.7,1.05) -- (1.7,0);
\draw (4.2,0) -- (5.8,0) -- (5.8,1.05) -- (4.2,1.05) -- (4.2,0);
\draw (6.7,0) -- (8.3,0) -- (8.3,1.05) -- (6.7,1.05) -- (6.7,0);

\draw (1.7,-2) -- (3.3,-2) -- (3.3,-0.95) -- (1.7,-0.95) -- (1.7,-2);
\draw (4.45,-2) -- (5.55,-2) -- (5.55,-0.95) -- (4.45,-0.95) -- (4.45,-2);

\end{tikzpicture}
\caption{Strategy for Alice demonstrating $\ch^{OL}(K_{4*3})\ge5$.\label{figure2}}
\end{figure}
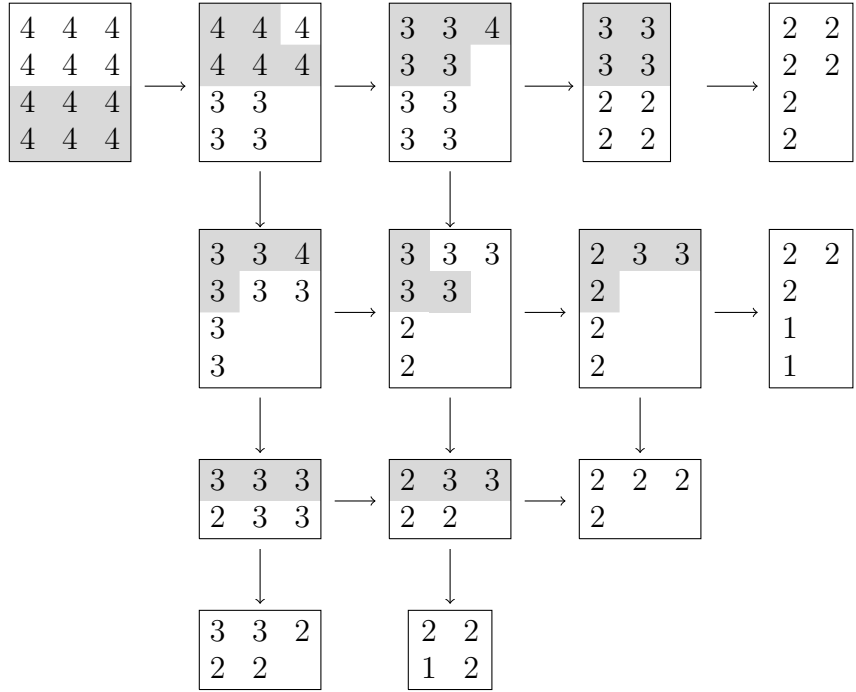

\begin{proof}
Figure~\ref{figure2} 
describes a strategy for Alice.  The top left matrix depicts the initial game position, and Alice's first move. The positions in the matrix correspond to the vertices of $K_{4*3}$ arranged so that vertices in the same part correspond to positions in the same column.  The order of vertices within a column is irrelevant, as is the order of the columns.  The numbers represent the size of the list of each corresponding vertex. The sequence of numbers represents a function $f$. The shaded positions represent the vertices that Alice presents on here first move.  

As play progresses Bob chooses certain vertices presented by Alice and passes over others.  When a vertex is chosen its position is removed from the next matrix (and the positions in its column of the remaining vertices and the order of the columns may be rearranged). When he passes over a vertex its list size is decreased by one (and its position in its column and the order of the columns may change). The arrows between the matrices  
point to the possible new game positions that arise from Bob's choice, not counting 
equivalent positions and omitting clearly inferior positions for Bob. In particular we assume Bob always chooses a maximal independent set.

For example, after Bob's first move there is only one possible game position, provided Bob chooses a maximal independent set. It is shown in the second column of the first row, along with Alice's second move. Now Bob has two possible responses that are pointed to by two arrows. Also consider the matrix in the third row and third column. There are three nonequivalent responses for Bob, but choosing the offered vertex in the second column of the matrix results in a position that is inferior to choosing the offered vertex in the first column. So this option is not shown.

Eventually, Alice forces one of five positions $(G,f)$ such that $G$ is not $f$-choosable, and Bob, being a gentleman, resigns. 
\end{proof}

\bibliographystyle{amsplain}
\bibliography{ch_citation}

\end{document}